\documentclass[11pt,twoside]{amsart}
\usepackage[latin1]{inputenc}
\usepackage[T1]{fontenc}
\usepackage{mathtools}
\usepackage{graphicx}
\usepackage{tikz}
\usetikzlibrary{chains}

\usepackage[latin1]{inputenc}
\usepackage{amsmath}
\usepackage{amsthm}
\usepackage{amssymb}
\usepackage[all]{xy}
\usepackage{hyperref}
\newtheorem{thm}{Theorem}[section]

\newtheorem{prop}[thm]{Proposition}

\newtheorem{lem}[thm]{Lemma}
\newtheorem{cor}[thm]{Corollary}

\numberwithin{equation}{section}

\theoremstyle{definition}

\newtheorem{remark}[thm]{Remark}

\newcommand{\qqed}{\hspace*{\fill}$\Box$}

\newcommand{\Db}{{\rm D}^{\rm b}}

\newcommand{\Mot}{{\rm Mot}}
\newcommand{\mot}{{\rm mot}}

\newcommand{\CH}{{\rm CH}}
\newcommand{\NS}{{\rm NS}}

\newcommand{\rk}{{\rm rk}}

\newcommand{\Spec}{{\rm Spec}}

\newcommand{\cal}{\mathcal}

\newcommand{\ke}{{\cal E}}

\newcommand{\LL}{\mathbb{L}}
\newcommand{\ZZ}{\mathbb{Z}}
\newcommand{\QQ}{\mathbb{Q}}

\newcommand{\FF}{\mathbb{F}}

\newcommand{\PP}{\mathbb{P}}

\newcommand{\hh}{\mathfrak{h}}

\renewcommand{\to}{\xymatrix@1@=15pt{\ar[r]&}}
\renewcommand{\rightarrow}{\xymatrix@1@=15pt{\ar[r]&}}
\renewcommand{\mapsto}{\xymatrix@1@=15pt{\ar@{|->}[r]&}}
\renewcommand{\twoheadrightarrow}{\xymatrix@1@=18pt{\ar@{->>}[r]&}}
\renewcommand{\hookrightarrow}{\xymatrix@1@=15pt{\ar@{^(->}[r]&}}
\newcommand{\hook}{\xymatrix@1@=15pt{\ar@{^(->}[r]&}}
\newcommand{\congpf}{\xymatrix@1@=15pt{\ar[r]^-\sim&}}
\renewcommand{\cong}{\simeq}

\begin{document}

\title[]{Motives of derived equivalent K3 surfaces}

\author[D.\ Huybrechts]{D.\ Huybrechts}

\address{Mathematisches Institut,
Universit{\"a}t Bonn, Endenicher Allee 60, 53115 Bonn, Germany}
\email{huybrech@math.uni-bonn.de}

\begin{abstract} \noindent
We observe that derived equivalent K3 surfaces have isomorphic Chow motives.
 \vspace{-2mm}
\end{abstract}

\maketitle

{\let\thefootnote\relax\footnotetext{The author is supported by the SFB/TR 45 `Periods,
Moduli Spaces and Arithmetic of Algebraic Varieties' of the DFG
(German Research Foundation).}
\marginpar{}
}

There are many examples of non-isomorphic K3 surfaces $S$ and $S'$ (over a field $k$)
with equivalent derived categories   $\Db(S)\cong\Db(S')$ of coherent sheaves. It is known that in this case
$S'$ is  isomorphic to a moduli space of slope-stable bundles on $S$ and vice versa, cf. \cite{HuyJAG,Mu,Or}.
However, the precise relation between $S$ and $S'$ still eludes us and
many basic questions are hard to answer. For example, one could ask whether $S'$ contains
infinitely many rational curves, expected for all K3 surfaces, if  this is known already for $S$. Or, if $k$ is a number field,
does potential density for  $S$ imply potential density for $S'$?

As a direct geometric approach to these questions seems difficult, one may wonder whether 
$S$ and $S'$ can be compared at least motivically.  This could mean various things, e.g.\
one could study their classes $[S]$ and $[S']$ in the Grothendieck  ring of varieties $K_0({\rm Var}(k))$ 
or  their associated motives $\hh(S)$ and $\hh(S')$ in the category of Chow motives $\Mot(k)$.
In the recent articles \cite{HL,IMOU,KS}, examples in degree $8$ and $12$ have been studied
for which $[S]-[S']$ in $K_0({\rm Var}(k))$ is annihilated by $\ell\coloneqq[{\mathbb A}^1]$. The question
whether the Chow motives $\hh(S)$ and $\hh(S')$ in $\Mot(k)$
are isomorphic was first addressed and answered affirmatively in special cases in \cite{Orlov}.
Assuming finite-dimensionality of the motives the question was settled in \cite{DPPed}.

\medskip

In this short note we point out that the available techniques in the theory of motives are enough
to show that two derived equivalent K3 surfaces have indeed isomorphic Chow motives.

\begin{thm}\label{thm:main}
Let $S$ and $S'$ be K3 surfaces over an algebraically closed field $k$. Assume that there
exists an exact $k$-linear equivalence $\Db(S)\cong\Db(S')$ between their bounded derived categories of coherent sheaves.
Then there is an isomorphism $$\hh(S)\cong\hh(S')$$
in the category of Chow motives $\Mot(k)$.
\end{thm}

The assumption on the field $k$ can be weakened, it suffices to assume
that $\rho(S)=\rho(S_{\bar k})$. 

\newpage

We had originally expected that the invariance of the Beauville--Voisin ring as
proved in \cite{HuyJEMS,HuyMSRI} would be central to the argument However, it turns to out to have no bearing on the problem, but
it implies that a distinguished decomposition of the motives in their algebraic and transcendental parts is preserved under derived equivalence.
\vskip0.6cm

\noindent
{\bf Acknowledgements:} I am very grateful to Charles Vial for answering my questions and to him and Andrey Soldatenkov for comments  on a first version and helpful suggestions.

\section{From derived categories to Chow motives}

We follow the convention and notation in \cite{A,MNP} and denote by $\Mot(k)$ the pseudo-abelian
category of Chow motives. Objects of this category are triples $M=(X,p,m)$, where $X$ is a
smooth projective $k$-variety of dimension $d_X$, $p\in\CH^{d_X}(X\times X)\otimes\QQ$ with $p\circ p=p$, and $m\in \ZZ$.
Morphisms $f\colon M=(X,p,m)\to N=(Y,q,n)$ are elements $\alpha\in\CH^{d_X+n-m}(X\times Y)\otimes\QQ$ such that $q\circ \alpha=\alpha=\alpha\circ p$.

The Chow motive of a smooth projective variety $X$ is denoted
$\hh(X)\coloneqq(X,[\Delta_X])\coloneqq(X,[\Delta_X],0)$.
The Lefschetz motive $\LL\coloneqq (\PP^1,\PP^1\times {\rm pt},0)\cong({\rm Spec}(k),{\rm id},-1)$
plays a distinguished role, e.g.\ the Chow groups of a motive $M=(X,p,m)$ are defined to be 
$\CH^i_\QQ(M)\coloneqq{\rm Mor}(\LL^{\otimes i},M)=p\circ \CH^i(M)\otimes\QQ$.

\subsection{} The following  result combines \ \cite[Lem.\ 1]{GG} and \cite[Lem.\ 3.2]{Vial}, cf.\ \cite[Lem.\ 4.3]{Ped}.

\begin{lem}\label{lem:folk}
Let $f\colon M\to N$ be a morphism in the category $\Mot(k)$ of Chow motives over a field $k$.
Then $f$ is an isomorphism  if and only if for all field extensions $K/k$ the
map $f_{K*}\colon\CH^*_\QQ(M_K)\congpf\CH^*_\QQ(N_K)$ induced by the base change
$f_K\colon M_K\to N_K$ is bijective. \qed
\end{lem}

One direction is obvious and for the other one use Manin's identity principle and induction over the dimension.

\subsection{} We start by describing  Orlov's approach \cite{Or}.
Let $\Phi\colon\Db(S)\congpf\Db(S')$ be an exact $k$-linear
equivalence between the bounded derived categories of coherent sheaves on two
K3 surfaces $S$ and $S'$. Then $\Phi$ is a Fourier--Mukai functor, i.e.\ $\Phi\cong\Phi_\ke\colon
F\mapsto p_*(q^*F\otimes\ke)$ for some $\ke\in\Db(S\times S')$  (unique up to isomorphism), see \cite{Orlov}. The Mukai vector 
$$v(\ke)\coloneqq {\rm ch}(\ke).\sqrt{{\rm td}(S\times S')}\in\CH^*(S\times S')\otimes\QQ$$
induces an (ungraded) isomorphism $v(\ke)_*\colon\CH^*(S)\congpf\CH^*(S')$, cf.\ \cite[Ch.\ 16]{HuyK3}.\footnote{For K3 surfaces
the isomorphism is indeed between integral Chow groups, but this is not true for other types of varieties.}

The inverse of $\Phi$ can be described as $\Phi^{-1}\cong\Phi_{\ke^*[2]}$, where $\ke^*$ is the derived dual of $\ke$. In particular,
$v(\ke^*)\circ v(\ke)=[\Delta_S]$ and $v(\ke)\circ v(\ke^*)=[\Delta_{S'}]$. Thus, if $v(\ke)$ is purely of degree two, i.e.\
$v(\ke)=v_2(\ke)\in\CH^2(S\times S')\otimes\QQ$,
then $v(\ke)$ can be considered as a morphism in $\Mot(k)$ and defines
an isomorphism $$v(\ke)\colon\hh(S)\congpf\hh(S').$$
However, a priori there is no reason why for two derived equivalent K3 surfaces $S$ and $S'$ there should
always exist an equivalence $\Phi_\ke\colon\Db(S)\congpf\Db(S')$ with $v(\ke)=v_2(\ke)$.
In fact,
when $S'$ is viewed as a moduli space of slope-stable bundles the universal family $\ke$ will certainly not
have this property, for $v_0(\ke)=\rk(\ke)$.

Next, one could try to just work with the degree two part $v_2(\ke)\in\CH^2(S\times S')\otimes\QQ$ of $v(\ke)$.
It defines a morphism $\hh(S)\to\hh(S')$ in $\Mot(k)$, which, however, will usually not be an isomorphism, for
the induced action on cohomology $v_{2*}^H\colon H^*(S)\to H^*(S')$ is often not bijective.

\subsection{} Instead of producing an isomorphism $\hh(S)\congpf\hh(S')$ directly, we shall first decompose
both motives with respect to their degree. The decompositions depend on the choice
of  additional cycles of degree one $c\in\CH^2(S)\otimes\QQ$ and $c'\in\CH^2(S')$, respectively, cf.\
\cite{A,MNP}. For $S$ it
reads \begin{equation}\label{eqn:hh}
\hh(S)\cong\hh^0(S)\oplus\hh^2(S)\oplus\hh^4(S),
\end{equation}
where $\hh^0(S)\coloneqq (S,p_0=[c\times S])$, $\hh^4(S)\coloneqq (S,p_4=[S\times c])$,
and $\hh^2(S)\coloneqq (S,p_2\coloneqq[\Delta_S]-p_0-p_4)$. 
Note that $\hh^0(S)\cong\hh(\Spec(k))\cong\LL^0$ and $\hh^4(S)\cong\LL^2$. Furthermore,
$$\CH^*_\QQ(\hh^0(S))=\CH^0(S)\otimes\QQ,~~\CH^*_\QQ(\hh^4(S))=\QQ\cdot c\subset\CH^2(S)\otimes\QQ,$$
$$\text{ 
 and }\CH^*_\QQ(\hh^2(S))=(\CH^1(S)\oplus\CH^2(S)_0)\otimes\QQ,$$
where $\CH^2(S)_0\coloneqq{\rm Ker}(\deg\colon \CH^2(S)\to\ZZ)$.

Following \cite{KMP}, the degree two part $\hh^2(S)$ can be further decomposed into its algebraic and
transcendental parts
\begin{equation}\label{eqn:hhtr}
\hh^2(S)\cong\hh^2_{\rm alg}(S)\oplus\hh^2_{\rm tr}(S).
\end{equation}
Here, $\hh^2_{\rm alg}(S)=(S,p_{2\rm alg})$ and $\hh^2_{\rm tr}(S)=(S,p_{\rm tr})$ with
$p_{2\rm alg}\coloneqq\sum 1/(\ell_i)^2 \ell_i\times\ell_i$, for  an orthogonal basis $\{\ell_i\}$ of
$\NS(S)\otimes\QQ$, and $p_{\rm tr}\coloneqq 1-p_{2\rm alg}-p_0-p_4$.
Note that
 $$\CH^*_\QQ(\hh^2_{\rm tr}(S))=\CH^2(S)_0\otimes\QQ  \text{ and } \CH^*_\QQ(\hh^2_{\rm alg}(S))=\CH^1(S)\otimes\QQ\cong{\rm NS}(S)\otimes\QQ$$
and, in fact, $\hh^2_{\rm alg}(S)\cong\LL^{\oplus\rho(S)}$.

 We shall call $\hh_{\rm alg}(S)\coloneqq\hh^0(S)\oplus\hh^2_{\rm alg}(S)\oplus\hh^4(S)=(S,p_{\rm alg}:=p_0+p_4+p_{2\rm alg})$
 the \emph{algebraic motive} of $S$. Then 
\begin{equation}\label{eqn:hhalg}
\hh_{\rm alg}(S)\cong\LL^0\oplus\LL^{\oplus \rho(S)}\oplus\LL^2
\end{equation}
and $$\hh(S)\cong\hh_{\rm alg}(S)\oplus\hh^2_{\rm tr}(S).$$

Next, under a field extension $K/k$ satisfying  $\rho(S)=\rho(S_K)$, one has
$\hh_{\rm alg}(S)_K\cong \hh_{\rm alg}(S_K)$
and, therefore, $\hh^2_{\rm tr}(S)_K\cong\hh^2_{\rm tr}(S_K)$.
In particular,  $\CH^*(\hh^2_{\rm tr}(S)_K)=\CH^2(S_K)_0\otimes\QQ$.

\subsection{} As  derived equivalent K3 surfaces have equal Picard numbers,
by (\ref{eqn:hhalg}) their algebraic motives are abstractly
(i.e.\ without choosing an actual equivalence)  isomorphic:
$$\Db(S)\cong\Db(S')\Rightarrow\hh_{\rm alg}(S)\cong\hh_{\rm alg}(S').$$
Thus, in order to prove Theorem \ref{thm:main}, it suffices to show the following.

\begin{prop} If $S$ and $S'$ are derived equivalent K3 surfaces, then there exists
an isomorphism
$$\hh_{\rm tr}^2(S)\congpf\hh^2_{\rm tr}(S').$$
\end{prop}

\begin{proof}
Pick an equivalence $\Phi_\ke\colon\Db(S)\congpf\Db(S')$ and consider  the Mukai vector
$v\coloneqq v(\ke)=\oplus v_i(\ke)\in\bigoplus\CH^i(S\times S')\otimes\QQ$ of its kernel. The degree two part $v_2(\ke)$
induces a morphism $f\colon\hh(S)\to\hh(S')$, which
decomposes as
$$f=\left(\begin{matrix}f_{\rm tr}&f_{\rm alg}^{\rm tr}\\f_{\rm tr}^{\rm alg}&f_{\rm alg}\end{matrix}\right)\colon\hh^2_{\rm tr}(S)\oplus\hh_{\rm alg}(S)\to\hh^2_{\rm tr}(S')\oplus\hh_{\rm alg}(S').$$

We show that $f_{\rm tr}^{\rm tr}\colon\hh^2_{\rm tr}(S)\to\hh^2_{\rm tr}(S')$ is an isomorphism.
According to  Lemma \ref{lem:folk}, it suffices to show that for all field extensions $K/k$
the induced map $f^{\rm tr}_{{\rm tr}K*}\colon\CH^*(\hh^2_{\rm tr}(S)_K)\congpf\CH^*(\hh^2_{\rm tr}(S')_K)$
is bijective. The two  sides are the cohomologically trivial parts of $\CH^*(S)\otimes\QQ$ and $\CH^*(S')\otimes\QQ$,
respectively, and hence $f_{{\rm tr}K*}^{\rm tr}=f_{K*}=v_2(\ke_K)_*$. However, as the cohomologically trivial
part of the Chow ring of a K3 surface is concentrated in degree two, the action  of $v_2(\ke_K)$ on $\CH^2(S_K)_0$ coincides with
the action of $v(\ke_K)$ on it. As $\ke_K$ is the base change of the Fourier--Mukai kernel of an equivalence,
it again describes an equivalence and, therefore, induces an ungraded isomorphism
$\CH^*(S_K)\otimes\QQ\congpf\CH^*(S'_K)\otimes \QQ$. This shows the bijectivity of $f_{{\rm tr}K*}^{\rm tr}$.
\end{proof}

\begin{remark} 
Assume the classes $c\in\CH^2(S)\otimes\QQ$ and $c'\in\CH^2(S')\otimes\QQ$ are the Beauville--Voisin classes
of $S$ and $S'$, respectively. Then, using \cite{HuyK3}, we know that $v(\ke)$ induces an ungraded isomorphism
$R(S)\cong R(S')$ (of groups) between the Beauville--Voisin rings. Here, $R(S)\subset\CH^*(S)\otimes\QQ$
is the $\QQ$-subalgebra generated by $\CH^1(S)$. We show that this implies that $f$ is in fact
diagonal
$$f=\left(\begin{matrix}f_{\rm tr}&0\\0&f_{\rm alg}\end{matrix}\right),$$ i.e.\ $f_{\rm tr}^{\rm alg}=0$ and $f_{\rm alg}^{\rm tr}=0$. Clearly,
$f_{{\rm alg}K*}^{{\rm tr}}\colon\CH^*(S_K)_0\otimes\QQ\to R(S'_K)\,\hookrightarrow H^*(S')$ is trivial
and, using \cite{HuyK3}, also $f_{{\rm tr}K*}^{\rm alg}\colon R(S)\to\CH^2(S_K)_0$ vanishes. However, to conclude
one needs to show that $p'_{\rm alg}\circ v_2\circ p_{\rm tr}=0=p'_{\rm tr}\circ v_2\circ p_{\rm alg}$ and both are implied
by $p'_{\rm alg}\circ v_2=v_2\circ p_{\rm alg}$. The latter holds in $H^*(S\times S')$ and as 
both sides are contained in $R(S)\otimes R(S')$ which injects into $H^*(S\times S')$, this suffices to conclude.


\end{remark}
\section{Applications and further comments}

We conclude by a few consequences and possible generalizations of the result.

\subsection{} Assume the two derived equivalent K3 surfaces $S$ and $S'$ are defined over a finite field $\FF_q$. 
Then due to \cite{LO}, using crystalline cohomology, and to \cite[Prop.\ 16.4.6]{HuyK3}, using \'etale cohomology, the Zeta functions of $S$ and $S'$ are shown to coincide, $Z(S,t)=Z(S',t)$.
In particular, $$\Db(S)\cong\Db(S')\Rightarrow |S(\FF_q)|=|S'(\FF_q)|.$$

This can now also be seen as a consequence of Theorem \ref{thm:main}. 

\subsection{} Another approach to capture the motivic nature of a K3 surface $S$ is to consider
its class $[S]\in K_0({\rm Var}(k))$ in the Grothendieck ring of varieties.
To  compare $[S]$ with the Chow motive $\hh(S)\in\Mot(k)$ or its class $[\hh(S)]$ in the Grothendieck
ring $$K(\Mot(k))\coloneqq\ZZ[\Mot(k)]/([M\oplus N]-[M]-[N]),$$ one uses the 
motivic Euler characteristic with compact support $\chi_\mot^c$, see  \cite{GS,GNA}
or \cite{B} for a much easier proof in characteristic zero. This
is a ring homomorphism
$$\chi_\mot^c\colon K_0({\rm Var}(k))\to K(\Mot(k)),$$
satisfying
$\chi^c_\mot([X])=[\hh(X)]$ for any smooth projective variety $X$
and $\chi_\mot^c(\ell)=[\LL]$, where $\ell\coloneqq[{\mathbb A}^1]$.

Theorem \ref{thm:main} shows that $\hh(S)\cong\hh(S')$ in $\Mot(k)$ for derived equivalent K3 surfaces
$S$ and $S'$ and, therefore, $[\hh(S)]=[\hh(S')]$ in $K(\Mot(k))$. This immediately yields

\begin{cor} Let $S$ and $S'$ be two derived equivalent K3 surfaces and let $[S]$ and $[S']$
be their classes in $K_0({\rm Var}(k))$. Then
$$\chi_\mot^c([S])=\chi_\mot^c([S'])\in K(\Mot(k)).$$
\vskip-0.8cm\qqed
\end{cor}

As $\chi_\mot^c$ factorizes via the localization
$K_0({\rm Var}(k))[\ell^{-1}]$, the corollary would also follow from
$[S]=[S']$ in $K_0({\rm Var}(k))[\ell^{-1}]$, i.e.\ 
$([S]-[S'])\cdot\ell^n=0$ in $K_0({\rm Var}(k))$ for some $n>0$. In fact,
in  the examples studied in \cite{HL,IMOU,KS},  dealing with K3 surfaces  of degree $8$ and $12$,
$([S]-[S'])\cdot \ell=0$. Note that $[S]=[S']$ in $K_0({\rm Var}(k))$\footnote{or, weaker,
in $K_0({\rm Var}(k))/(\ell)$ which in characteristic zero
is isomorphic to the Grothendieck ring of stable birational classes
of smooth projective varieties $\ZZ[{\rm SB}(k)]$, see \cite{LL}.} if and only if $S\cong S'$, see \cite{LS}.
In particular, the kernel of $$K_0({\rm Var}(k))\to K(\Mot(k))$$ is certainly non-trivial. However, to the best
of my knowledge, nothing is known about the kernel of $$K_0({\rm Var}(k))[\ell^{-1}]\to K(\Mot(k)).$$

\subsection{} Theorem \ref{thm:main} clearly implies that for two derived equivalent K3 surfaces $S$ and
$S'$, one is finite-dimensional (in the sense of Kimura--O'Sullivan) if and only if the other one is. This was
already observed in \cite[Prop.\ 1.5]{DPPed}. In fact, in this paper Theorem \ref{thm:main}
was shown assuming finite-dimensionality
of $\hh(S)$, which however is only known in very few cases.

\subsection{} As Charles Vial explained to me, Theorem \ref{thm:main} can be
generalized to arbitrary  surfaces. Indeed, it can be
shown that the odd degree part $\hh^1\oplus\hh^3$ is a derived invariant. 
For details see \cite{Honig} and the forthcoming paper of Achter,
Casalaina-Martin, and Vial, where the more complicated
situation of derived equivalences of threefolds is studied.

He also explained to me that Lemma \ref{lem:folk} can be avoided by showing
that the isomorphism of ungraded Chow motives induced by $v(\ke)$  necessarily maps
$\hh_{\rm tr}^2(S)$ into $\hh_{\rm tr}^2(S')$. The inverse is then given by $v(\ke^*)$.


\end{document}